\documentclass[11pt,reqno]{amsart}
\usepackage{amssymb}
{\theoremstyle{plain}
 \newtheorem{theorem}{Theorem}
 \newtheorem{corollary}{Corollary}
 
\newtheorem{lemma}{Lemma}
}

\DeclareMathOperator{\re}{Re}
\DeclareMathOperator{\im}{Im}

\begin{document}

\title{An extension of the Siegel-Walfisz theorem}

\begin{abstract}
We extend the Siegel-Walfisz theorem to a family of integer sequences that 
are characterized by constraints on the size of the prime factors. 
Besides prime powers, this family includes smooth numbers, almost primes and practical numbers. 
\end{abstract}

\author{Andreas Weingartner}
\address{ 
Department of Mathematics,
351 West University Boulevard,
 Southern Utah University,
Cedar City, Utah 84720, USA}
\email{weingartner@suu.edu}
\date{March 28, 2021}
\maketitle

\section{Introduction}

Let $\pi(x,q,a)$ denote the number of primes up to $x$ that are congruent to $a \bmod q$
and let $\pi(x)$ be the number of all primes up to $x$. 
The Siegel-Walfisz theorem shows that, for any fixed $A>0$, we have
\begin{equation}\label{SWO}
\pi(x,q,a)
=\frac{\pi(x)}{\varphi(q)}+ O_A\left(\frac{x}{(\log x)^A}\right) ,
\end{equation}
uniformly for $2\le q\le (\log x)^A$ and $\gcd(a,q)=1$.
Thus, the primes are close to evenly distributed among the $\varphi(q)$ residue classes that are 
coprime to the modulus $q$. 

We will generalize \eqref{SWO} to a 
family of integer sequences that arise as follows. 
Let $\theta$ be an arithmetic function that satisfies
\begin{equation}\label{thetadef}
\theta : \mathbb{N} \to \mathbb{R}\cup \{\infty\}, \quad \theta(1)\ge 2, \quad \theta(n)\ge P^+(n) \quad (n\ge 2),
\end{equation}
where $P^+(n)$ denotes the largest prime factor of $n$. 
Let $\mathcal{B}=\mathcal{B}_\theta$ be the set of positive integers containing $n=1$ and all those $n \ge 2$ with prime factorization $n=p_1^{\alpha_1} \cdots p_k^{\alpha_k}$, $p_1<p_2<\ldots < p_k$, which satisfy 
\begin{equation*}\label{Bdef}
p_{i} \le \theta\big( p_1^{\alpha_1} \cdots p_{i-1}^{\alpha_{i-1}} \big) \qquad (1\le i \le k).
\end{equation*}

\renewcommand{\arraystretch}{1.2}
\begin{table}[h]
     \begin{tabular}{ | c |  c |  }
    \hline
    $ \theta(n) $ & $\mathcal{B} $\\ \hline   
    $
     \begin{cases}
   \infty & n=1 \\
  P^+(n)  & n\ge 2
   
  \end{cases}
  \quad \ 
     $ 
     & prime powers \\ \hline   
    $
     \begin{cases}
   \infty & \omega(n)<k \\
  P^+(n)  & \omega(n)\ge k
   
  \end{cases}
     $ 
     & almost primes ($\omega(n)\le k$) \\ \hline   
     $\max\left(y,P^+(n)\right)$ & $y$-smooth numbers  \\ \hline
     $nt$ & $t$-dense numbers \\ \hline
     $\sigma(n)+1$ & practical numbers \\ \hline
\end{tabular}        
   \medskip
     \caption{Examples of $\theta$ and the corresponding set $\mathcal{B}$.}\label{table0}
\end{table}

\renewcommand{\arraystretch}{1}

Let $\mathcal{B}(x) = \mathcal{B}\cap [1,x]$ and $B(x) =|\mathcal{B}(x)|$. 
We will only use the subscript $\theta$ to avoid confusion when more than one $\theta$ is involved in the discussion.
Note that the two conditions on $\theta$ in \eqref{thetadef} are for convenience only and do not represent a restriction, because $\theta(1)<2$
implies $\mathcal{B}=\{1\}$, and replacing $\theta(n)$ by $\max(\theta(n),P^+(n))$ does not change $\mathcal{B}$. 

Several examples of $\theta$ and the corresponding set $\mathcal{B}$ are listed in Table \ref{table0}, where $\omega(n)$ is the number of distinct prime factors of $n$ and $\sigma(n)$ is the sum of the positive divisors of $n$. 
The $t$-dense numbers \cite{Saias1,Ten86,PDD} are integers whose ratios of consecutive divisors do not exceed $t$. 
The practical numbers \cite{Mar, Saias1, Sri, Ten86,PDD} are integers $n$ such that every natural number $m$ up to $n$ can be expressed as a sum of distinct positive divisors of $n$.

Theorem \ref{thm1} shows that the integers in $\mathcal{B}$
are close to evenly distributed among all the residue classes that are coprime to the modulus,
provided the main term is not absorbed by the error term.
If $\theta(1)=\infty$ and $\theta(n)=P^+(n)$ for $n\ge 2$, then $\mathcal{B}$ is the set of prime powers and
Theorem \ref{thm1} reduces to \eqref{SWO}, since higher powers of primes, as well as primes dividing $q$, are negligible. 
The Siegel-Walfisz theorem, in the form of Lemma \ref{SW}, is a key ingredient in the proof of Theorem \ref{thm1}.
\begin{theorem}\label{thm1}
Assume $\theta$ satisfies \eqref{thetadef}. Let $A>0$ be fixed. Uniformly for $2\le q\le (\log x)^A$, $a\in \mathbb{Z}$ with
$\gcd(a,q)=1$, we have
$$
B(x,q,a):=\sum_{n \in \mathcal{B}(x) \atop n\equiv a \bmod q} 1
=  O_A\left(\frac{x}{(\log x)^A}\right) + \frac{1}{\varphi(q)} \sum_{n \in \mathcal{B}(x) \atop \gcd(n,q)=1} 1 .
$$
The implied constant only depends on $A$ and not on the choice of $\theta$. 
\end{theorem}

For the problem of smooth numbers in arithmetic progressions,  
Harper \cite{Harper} establishes asymptotic equidistribution among the residue classes
in a larger domain than is implied by Theorem \ref{thm1}. 
For almost primes ($\omega(n)\le k$) in arithmetic progressions, Friedlander \cite{Fried} ($k=2$) and 
Knafo \cite[Thm. 6]{Knafo} ($k \ge 2$) give estimates for the corresponding sums with weights $\Lambda_k(n)$, 
the generalized von Mangoldt function. 

Unlike with primes, the distribution of integers in $\mathcal{B}$ among residue classes $a\bmod q$ with $\gcd(a,q)>1$
is not immediately obvious in general. If $\theta $ satisfies $\theta(n)\le \theta(mn)$ for all $m,n \in \mathbb{N}$, we 
will show that Theorem \ref{thm1} implies uniform distribution among all the residue classes that share the same greatest common divisor with the modulus.
\begin{corollary}\label{cor0}
Assume $\theta$ satisfies \eqref{thetadef} and $\theta(n)\le \theta(mn)$ for all $m,n \in \mathbb{N}$. Let $A>0$ be fixed. Uniformly for $2\le q\le (\log x)^A$, $a\in \mathbb{Z}$, $d:=\gcd(a,q)$, we have
$$
B(x,q,a)
= O_A\left(\frac{x}{(\log x)^A}\right)+\frac{1}{\varphi(q/d)} \sum_{n \in \mathcal{B}(x) \atop \gcd(n,q)=d} 1 .
$$
The implied constant only depends on $A$ and not on the choice of $\theta$. 
\end{corollary}

Define
$$
\mathcal{B}_q(x) =\mathcal{B}_{\theta,q} (x) = \{n\in \mathcal{B}_\theta(x), q|n\}, \quad B_q(x) =B_{\theta,q}(x) = |\mathcal{B}_q(x) |.
$$
The inclusion-exclusion principle allows us to rewrite the right-hand side of Corollary \ref{cor0} as follows.
\begin{corollary}\label{cor1}
With the assumptions of Corollary \ref{cor0} we have 
$$
B(x,q,a)
= O_A\left(\frac{x}{(\log x)^A}\right)+
\frac{1}{\varphi(q/d)} \sum_{m|q/d} \mu(m) B_{dm}(x) ,
$$
where $d=\gcd(a,q)$. 
\end{corollary}
One way to understand $B_q(x)$, and thus the right-hand side of Corollary \ref{cor1}, 
is to consider a modified version of $\theta$, denoted by $\theta_q$, 
which we define as follows. 
Let $q=q_1 \cdots q_k$, where $q_1\le \ldots \le q_k$ are the prime factors of $q$.
If $\theta(n)<q_1$ then $\theta_q(n)=\theta(n)$. Otherwise,
\begin{equation*}\label{tildetheta}
\theta_q(n) = \theta(n q_1\cdots q_j)
\end{equation*}
where $j$ is as large as possible, subject to  
$$
q_i \le \theta(nq_1 \cdots q_{i-1}) \quad (1\le i \le j).
$$
\begin{lemma}\label{prop1}
Assume \eqref{thetadef} and $\theta(n)\le \theta(mn)$ for all $n,m\in \mathbb{N}$. We have  
$$ B_{\theta_q}(x/q) - R(x,q) \le  B_{\theta,q}(x) \le  B_{\theta_q}(x/q) ,$$
where
$$
R(x,q) = |\{n\le x/q: \theta(n)<P^+(q)\}|.
$$
\end{lemma}

Since $\theta(n)\ge P^+(n)$, we always have 
\begin{equation}\label{Req}
R(x,q) \le \Psi(x,q):=|\{n\le x: P^+(n)\le q\}| \ll_A \frac{x}{(\log x)^A}, 
\end{equation}
provided $q\le (\log x)^A$, by \cite[Thm. III.5.1]{Ten}. If $\theta(n)\ge n$, then $R(x,q) < P^+(q)$. 

With Lemma \ref{prop1}, the known asymptotic results for $B_\theta(x)$ under various conditions on $\theta(n)$  (see \cite{PDD,SPA}),
applied with $\theta_q$ in place of $\theta$, translate to asymptotics for $B_{\theta,q}(x)$,
which in turn lead to  estimates for $B(x,q,a)$, by Corollary \ref{cor1}.  
Since the implied constants in the error terms of these estimates depend on $\theta$,
any estimates for $B(x,q,a)$ derived in this manner will have implied constants that depend on $q$. 
We only take a closer look at the case when $n\le \theta(n) \ll \sigma(n)$,
which includes the practical numbers as well as the $t$-dense integers when $t\ll 1$.

If $n \le \theta(n) \ll \sigma(n)$ and $\theta(mn)\ll m^{O(1)}\theta(n)$, 
\cite[Thm. 1.2]{PDD} and \cite[Remark 1]{PW} show that
\begin{equation}\label{Basymp}
B(x) = \frac{c_\theta x}{\log x}\left(1+O_\theta \left(\frac{1}{\log x}\right)\right),
\end{equation}
for some positive constant $c_\theta$.
Lemma \ref{prop1} implies
$$
B_{\theta,q}(x) = \frac{c_{\theta_q} x/q}{\log x/q}\left(1+O_{\theta, q}\left(\frac{1}{\log x/q}\right)\right)
=  \frac{c_q x}{\log x}\left(1+O_{\theta,q}\left(\frac{1}{\log x}\right)\right),
$$
where $c_q:=c_{\theta_q} /q$. Thus, Corollary \ref{cor1}, Lemma \ref{prop1} and \eqref{Basymp}
yield the following estimate for $B(x,q,a)$. 
 
\begin{corollary}\label{cor3}
Assume \eqref{thetadef}, $n\le \theta(n) \ll \sigma(n)$ and $\theta(n)\le \theta(nm)\ll m^{O(1)}\theta(n)$ for all $n,m \in \mathbb{N}$. 
For $q\in \mathbb{N}$, $a \in \mathbb{Z}$ and $d=\gcd(q,a)$, we have
\begin{equation}\label{BxqaAsymp}
B(x,q,a) = \frac{c_{q,a} x}{\log x}  +O_{\theta,q}\left(\frac{x}{(\log x)^2}\right),
\end{equation}
where
\begin{equation}\label{cqa}
c_{q,a}=\frac{1}{\varphi(q/d)} \sum_{m|q/d} \mu(m) c_{dm},\quad  d=\gcd(a,q),
\end{equation}
$$
c_{dm}:=c_{dm,0}=\frac{c_{\theta_{dm}}}{dm}, \quad 
c_{\theta_{dm}} = \lim_{x\to \infty} \frac{B_{\theta_{dm}}(x)}{x/\log x}.
$$
Moreover, $c_{q,a}=0$ if and only if $B(x,q,a)\le 1$ for all $x$. 
\end{corollary}

Margenstern \cite[Prop. 6 and p. 17]{Mar} showed that the number of practical $n$ with $n \equiv a \bmod q$ is either
$0$ (e.g. $10\bmod 12$) or $1$ (e.g. $2 \bmod 12$) or $\infty$ (e.g. $4\bmod 12$). 
Moreover, if there is only one solution, it must be less than $q$. 
The last statement of Corollary \ref{cor3}, which we prove in Section \ref{seccqa},
 says that if an arithmetic progression contains more than one practical number
(i.e. at least one that is $\ge q$), 
then it contains a positive proportion of all practical numbers, because $c_{q,a}>0$ in that case.  

Let $r_{q,a} := \lim_{x\to \infty} B(x,q,a)/B(x) = c_{q,a}/c_\theta$, the asymptotic ratio of members of $\mathcal{B}$
that are congruent to $a$ modulo $q$, and write $r_q=r_{q,0}$. Dividing \eqref{cqa} by $c_\theta$, we have
\begin{equation}\label{rqa}
r_{q,a}=\frac{1}{\varphi(q/d)} \sum_{m|q/d} \mu(m) r_{dm},\quad  d=\gcd(a,q).
\end{equation}
With the assumptions of Corollary \ref{cor3}, the constant factor $c_\theta$ in \eqref{Basymp} is given in
\cite[Thm. 1]{CFAE} as the sum of an infinite series,
$$
c_\theta = \frac{1}{1-e^{-\gamma}} \sum_{n\in \mathcal{B}} \frac{1}{n}  \Biggl( \sum_{p\le \theta(n)}\frac{\log p}{p-1} - \log n\Biggr) \prod_{p\le \theta(n)} \left(1-\frac{1}{p}\right),
$$
where $\gamma$ is Euler's constant and $p$ runs over primes.
In the case of $t$-dense integers \cite{CFAE} and practical numbers \cite{CFP}, we have methods for estimating 
this series for $c_\theta$ with good precision. The method in  \cite{CFP} works more generally if
$\theta(n)$ is close to a multiplicative function when $n$ is large. 
By making small adjustments to these algorithms, we can estimate the 
corresponding factors $c_{\theta_{q}}$, and therefore $c_q = c_{\theta_{q}}/q$ and $r_q = c_q/c_\theta$. 
The results of our calculations are shown in Table \ref{table1}.

\renewcommand{\arraystretch}{1.1}
\begin{table}[h]
     \begin{tabular}{ | r |  l |  l |  l | l | }
    \hline
    $q $ &\  \small{$\theta(n)=2n$} \  & \ \small{$\theta(n)=3n$}  \  & \ \small{$\theta(n)=5n$}\  & \small{$\theta(n)=\sigma(n)+1$} \\ \hline   
    $2$ & \ $1$ & \ $0.79003...$  & \ $0.71557...$  & \ \  $1$ \\ \hline
    $3$ & \ $0.63176...$  & \ $0.65544...$ &\  $0.57660...$ & \ \ $0.64880$ \\ \hline 
     $4$ & \ $0.78597... $  & \ $0.56470...$ & \ $0.48593...$ & \  \ $0.77728...$\\ \hline 
    $5$ & \ $0.38362$  & \ $0.41710...$  & \ $0.42042...$ & \  \ $0.38261...$ \\ \hline
    $6$ & \ $0.63176...$  &  \ $0.44548...$ &  \ $0.37177...$  & \ \ $0.64880$ \\ \hline 
    $7$ & \ $0.30335...$ & \ $0.29778...$   & \ $0.29217...$ & \ \ $0.29590...$   \\ \hline
    $8$ & \ $0.53410...$ &\ $0.37339...$  & \ $0.30509...$ & \ \ $0.52377...$ \\ \hline
    $9$ & \ $0.31635...$ &  \ $0.34353...$ & \ $0.28158...$  & \ \ $0.31603...$ \\ \hline
     $10$ &\  $0.38362$  & \  $0.32059...$ & \ $0.26100...$ & \  \ $0.38261...$ \\ \hline
    $11$ & \ $0.19841...$ & \  $0.19697...$  & \ $0.19088...$ & \ \ $0.19182...$  \\ \hline 
    $12$ &  \ $0.41774...$  &  \  $0.28277...$  & \ $0.22887...$  & \ \ $0.42608...$ \\ \hline 
    $13$ & \ $0.16292...$  &   \  $0.16279...$  & \ $0.16763...$  & \ \ $0.16786...$ \\ \hline 
    $14$ & \ $0.30335...$   &    \  $0.22041...$  & \ $0.20435...$  & \ \ $0.29590...$   \\ \hline
    $15$ & \ $0.22080...$ &     \  $0.24281...$  & \ $0.19443...$  & \ \ $0.22354...$ \\ \hline
    $16$ &\  $0.34407...$  &    \  $0.23200...$  & \ $0.18535...$  & \ \ $0.33425... $ \\ \hline
    $17$ &\  $0.12463...$  &   \  $0.13147...$ & \ $0.12813...$    & \ \ $0.12110...$ \\ \hline
    $18$ &\  $0.31635...$ &    \  $0.21317...$ & \ $0.16967...$   & \ \ $0.31603...$ \\ \hline
    $19$ &\  $0.11389...$ &      \  $0.11553...$  & \ $0.11713...$   & \ \ $0.11042...$ \\ \hline
    $20$ &\  $0.29434...$ &     \  $0.19736...$ & \ $0.15653$  & \ \ $0.29275...$ \\ \hline
     \end{tabular}        
   \medskip
     \caption{Truncated values of $r_q$, the asymptotic proportion of integers in $\mathcal{B}_\theta$ that are multiples of $q$. Values without ``...'' are rounded.}\label{table1}
\end{table}
\renewcommand{\arraystretch}{1}

Since all practical numbers greater than $1$ are even, their last decimal digit is determined by the remainder after dividing by $5$.
Table \ref{table1} shows that $r_5\approx 38.26\%$ of practical numbers have a last decimal digit of $0$, while the last digits of $2, 4, 6, 8$ each occur with a relative frequency of $(1-r_5)/4 \approx 15.43\%$, by \eqref{rqa}.

The distribution of $2$-dense integers modulo $12$ is given by Table \ref{table1} and \eqref{rqa} as 
$$
r_{12,0}=r_{12}\approx 0.4177, \quad
r_{12,6}=r_{6}-r_{12}\approx 0.2140, 
$$
$$
 r_{12,4}=r_{12,8}=\frac{r_4-r_{12}}{2}\approx 0.1841, \quad  r_{12,2}=r_{12,10}=0,
$$
and $r_{12,a}=0$ for odd $a$, where the vanishing ratios are a simple consequence of the definition of $2$-dense integers.

\section{Proof of Theorem \ref{thm1}}

Throughout this section, $\chi$ stands for a Dirichlet character  modulo $q$ and $\chi_0$ denotes the principal character. 
We write $s=\sigma + i \tau$, where $\sigma=\re(s)$ and $\tau = \im(s)$. 

\begin{lemma}\label{lemPrime}
We have
$$
\sum_{n\ge 2} \frac{1}{n (\log P^+(n))^{2}} \ll 1.
$$
\end{lemma}
\begin{proof}
The sum equals
$$
\sum_{p\ge 2} \frac{1}{p(\log p)^2} \sum_{P^+(m)\le p}\frac{1}{m}
= \sum_{p\ge 2} \frac{1}{p(\log p)^2} \prod_{q\le p} \left(1-\frac{1}{q}\right)^{-1} 
\ll \sum_{p\ge 2} \frac{1}{p\log p} \ll 1.
$$
\end{proof}

\begin{lemma}\label{lem1}
For $\gcd(a,q)=1$ we have
$$
\sum_{n \in \mathcal{B}(x) \atop n\equiv a \bmod q} 1
=\frac{1}{\varphi(q)} \sum_{n \in \mathcal{B}(x) \atop \gcd(n,q)=1} 1
+ \frac{1}{\varphi(q)}  \sum_{\chi \neq \chi_0} \overline{\chi}(a) \sum_{n\in \mathcal{B}(x)} \chi(n). 
$$
\end{lemma}
\begin{proof}
This follows from summing over $n\in \mathcal{B}(x)$ the orthogonality relation for Dirichlet characters \cite[Eq. 4.27]{MV}. 
\end{proof}

\begin{lemma}\cite[Sec. 11.3.1, Ex. 5]{MV}\label{SW}
Let $A>0$ be fixed  and $\chi \neq \chi_0$. For $q\le (\log x)^A$ we have
$$
\sum_{p\le x} \chi(p)  \ll_A \frac{x}{(\log x)^A}.
$$
\end{lemma}

\begin{lemma}\label{lempsum}
Let $A>0$ be fixed  and $\chi \neq \chi_0$. For $q\le (\log x)^A$  we have
$$
\sum_{p>x} \frac{\chi(p)}{p^s} \ll_{A} \frac{|\tau|+1 }{(\log x)^A} \qquad (\sigma \ge 1).
$$
\end{lemma}
\begin{proof}
If $1\le \sigma \le 2$, this follows from Lemma \ref{SW} and partial summation. If $\sigma >2$, it is obvious.
\end{proof}
\begin{lemma}\label{lemfunceq}
Assume $\theta$ satisfies \eqref{thetadef}.
Let $\chi$ be any Dirichlet character. For $\sigma>1$ we have
$$
L_\chi(s) :=\sum_{m\ge 1} \frac{\chi(m)}{m^s}= \sum_{n \in \mathcal{B}} \frac{\chi(n)}{n^s}\prod_{p>\theta(n)} \left(1-\frac{\chi(p)}{p^s}\right)^{-1}.
$$
\end{lemma}
\begin{proof}
Each $m \ge 1$ factors uniquely as $m=n r$, with $n\in \mathcal{B}$ and 
$P^-(r)>\theta(n)$ if $r>1$, where $P^-(r)$ denotes the smallest prime factor of $r$. 
\end{proof}

\begin{lemma}\label{lemL}
Let $\chi\neq \chi_0$ and $T\ge 2$. For $|\tau | \le T$ and $\sigma \ge 0$, 
$$ L_\chi(s) \ll (T q)^{\max(0,(1-\sigma)/2)} \log (T q).
$$
\end{lemma}
\begin{proof}
When $0\le \sigma \le 1$, this is \cite[Eq. (3)]{Kol}. If $\sigma \ge 1$, this follows from \cite[Thm. II.8.6]{Ten}
\end{proof}

\begin{proof}[Proof of Theorem \ref{thm1}]
Assume $\chi \neq \chi_0$. We have
$$
\sum_{n \in \mathcal{B}(x)} \chi(n) = \sum_{n\le x}\chi(n)- \sum_{n \le x \atop n \notin  \mathcal{B}(x)} \chi(n) 
=O(q) - S(x),
$$
say. Theorem \ref{thm1} follows from Lemma \ref{lem1} if we can show $S(x) \ll x/(\log x)^A$.
With $T=(\log x)^{A+1}$ and $\kappa = 1+1/\log x$, Perron's formula \cite[Cor. II.2.2.1]{Ten} yields
\begin{equation}\label{Seq}
S(x) = \frac{1}{2\pi i} \int_{\kappa - i T}^{\kappa + i T} F(s) x^s \frac{ds}{s} + O\left(\frac{x }{(\log x)^A}\right),
\end{equation}
where, for $\sigma >1$, 
$$
F(s) := \sum_{m\notin \mathcal{B}} \frac{\chi(m)}{m^s} 
= \sum_{n \in \mathcal{B}} \frac{\chi(n)}{n^s} 
\left(\prod_{p>\theta(n)} \left(1-\frac{\chi(p)}{p^s}\right)^{-1}   -1 \right),
$$
by Lemma \ref{lemfunceq}.
Let $F_1(s)$ (resp. $F_2(s)$) denote the contribution to the last sum from $n$ with $\theta(n)\le N$ (resp. $\theta(n)>N$),
where $N$ is given by $\log N = T^{1/\lambda}$ and $\lambda>0$ is constant. 

We first estimate the contribution of $F_2(s)$ to the integral in \eqref{Seq}. 
For any fixed $\mu\ge \lambda$ and all $n$ with $\theta(n)>N$, we have
$$
q \le T \le T^{\mu/\lambda} = (\log N)^\mu < (\log \theta(n))^\mu.
$$
If $|\tau|\le T$, Lemma \ref{lempsum} yields
$$
\prod_{p>\theta(n)} \! \left(1-\frac{\chi(p)}{p^s}\right)^{-1} \!\!\!
=\exp\left(\sum_{p>\theta(n)}\frac{\chi(p)}{p^s} + O\left(\frac{1}{\theta(n)}\right)\right)
\! =1+O\left(\frac{|\tau|+1}{(\log \theta(n))^\mu}\right).
$$
Thus, 
$$
F_2(s) \ll \sum_{\theta(n)>N} \frac{1}{n} \cdot \frac{|\tau|+1}{(\log \theta(n))^\mu}
\ll \frac{|\tau|+1}{(\log N)^{\mu-2}},
$$
by \eqref{thetadef} and Lemma \ref{lemPrime}. With $\mu=2\lambda+2$, the contribution of $F_2(s)$ to \eqref{Seq} is
$$
 \frac{1}{2\pi i} \int_{\kappa - i T}^{\kappa + i T} F_2(s) x^s \frac{ds}{s}
 \ll \frac{x T}{(\log N)^{\mu-2}}
 =\frac{x}{T^{(\mu-2)/\lambda -1}}
 =\frac{x}{T}
 =\frac{x}{(\log x)^{A+1}}.
$$

To estimate the contribution of $F_1(s)$ to \eqref{Seq}, we rewrite $F_1(s)$ as 
$$
F_1(s) = \sum_{n \in \mathcal{B} \atop \theta(n)\le N}
 \frac{\chi(n)}{n^s} 
\left(L_\chi (s) \prod_{p\le \theta(n)} \left(1-\frac{\chi(p)}{p^s}\right)   -1 \right).
$$
This formula, initially valid for $\sigma>1$, gives the analytic continuation of $F_1(s)$ to 
the half-plane $\sigma >0$, since 
$$
\prod_{p\le \theta(n)} \left|1-\frac{\chi(p)}{p^s}\right|  \le 2^{\theta(n)} \le 2^N
$$
and
$$
\sum_{n \in \mathcal{B} \atop \theta(n)\le N} \frac{|\chi(n)|}{n^\sigma} 
 \le \sum_{P^+(n)\le N} \frac{1}{n^\sigma} 
 =\prod_{p\le N}\left(1-\frac{1}{p^\sigma}\right)^{-1}.
$$
We replace the vertical line segment, with endpoints
$\kappa - iT$ and $\kappa +iT$, by the three line segments with endpoints $\kappa - iT$, $\alpha -iT$, $\alpha+iT$,
and $\kappa +iT$, where $\alpha = 1-1/(\log x)^\delta$ and $\delta = (A+1)/\lambda$. 
On each of  these three line segments, we have $\alpha \le \sigma \le \kappa$ and $|\tau|\le T$, so that 
$$
\prod_{p\le \theta(n)} \left(1-\frac{\chi(p)}{p^s}\right) 
\ll \exp\left( \sum_{p\le \theta(n)} \frac{1}{p^\alpha} \right)
\le \exp\left( \sum_{p\le N} \frac{N^{1-\alpha}}{p} \right),
$$
$$
 \sum_{n \in \mathcal{B} \atop \theta(n)\le N} \frac{|\chi(n)|}{n^\sigma} 
 \le \sum_{P^+(n)\le N} \frac{1}{n^\alpha} 
 =\prod_{p\le N}\left(1-\frac{1}{p^\alpha}\right)^{-1}
 \ll \exp\left( \sum_{p\le N} \frac{N^{1-\alpha}}{p} \right),
$$
$$
\exp\left( \sum_{p\le N} \frac{N^{1-\alpha}}{p} \right)
=\exp\bigl(N^{1-\alpha} (\log\log N +O(1))\bigr)
\ll (\log x)^{e(A+1)/\lambda}
$$
and
$$
L_\chi(s) \ll (Tq)^{1/(\log x)^{\delta}}\log (Tq) \ll \log\log x,
$$
by Lemma \ref{lemL}.
With $\lambda = 6(A+1)$, we get 
$$
F_1(s) \ll (\log x)^{2e(A+1)/\lambda} \log\log x \ll \log x.
$$
Since $\delta = (A+1)/\lambda = 1/6$, the contribution from the vertical segment is
$$
 \frac{1}{2\pi i} \int_{\alpha - i T}^{\alpha + i T} F_1(s) x^s \frac{ds}{s} 
 \ll x^\alpha (\log x) (\log T) \ll_A \frac{x}{(\log x)^A}. 
$$
The contribution from the two horizontal segments can be estimated as
$$
 \frac{1}{2\pi i} \int_{\alpha + i T}^{\kappa + i T} F_1(s) x^s \frac{ds}{s} \ll \frac{x \log x}{T}=\frac{x}{(\log x)^A}. 
$$
This completes the proof of Theorem \ref{thm1}.
\end{proof}

\section{Proof of Lemma \ref{prop1}}

\begin{lemma}\label{lemprop1}
Assume \eqref{thetadef} and $\theta(n)\le \theta(mn)$ for all $n,m\in \mathbb{N}$. 
We have  
\begin{equation}\label{sets}
\{m: mq \in \mathcal{B}_\theta \} \subset \mathcal{B}_{\theta_q} 
\subset \{m: mq \in \mathcal{B}_\theta \}  \cup \{m: \theta(m)<P^+(q)\}.
\end{equation}
\end{lemma}

Lemma \ref{prop1} follows from \eqref{sets} by counting the members of each set up to $x/q$. 

\begin{proof}[Proof of Lemma \ref{lemprop1}]

To show the first set inclusion in \eqref{sets}, let $mq\in \mathcal{B}$.
Let $m=p_1\cdots p_l$ and $q=q_1\cdots q_k$  be
the prime factorizations in increasing order. 
Let $i$ satisfy $1\le i \le l$ and define $j$ with $0\le j \le k$  by $q_j<p_i\le q_{j+1}$.
Since $mq\in \mathcal{B}$, we have
$$
p_i\le \theta(p_1 \cdots p_{i-1} q_1\cdots q_j).
$$
If $j=0$, then $p_i \le \theta(p_1 \cdots p_{i-1} ) \le \theta_q(p_1\cdots p_{i-1})$. 
If $j\ge 1$, then  for $1\le s\le j$, $mq\in \mathcal{B}$ implies
$$
q_s \le \theta\Bigl(q_1\cdots q_{s-1} \prod_{p_t < q_s}p_t \Bigr)
\le  \theta\Bigl(q_1\cdots q_{s-1} \prod_{p_t < q_j}p_t \Bigr)
\le  \theta(q_1\cdots q_{s-1} p_1\cdots p_{i-1}).
$$
This shows that $\theta_q(p_1\cdots p_{i-1})\ge \theta(p_1 \cdots p_{i-1} q_1\cdots q_j) \ge p_i$,
so $m \in \mathcal{B}_{\theta_q} $.

To show the second set inclusion in \eqref{sets}, assume that $mq \notin \mathcal{B}$. 
Then there exists and index $i$ such that either
\begin{equation}\label{vio1}
q_i>\theta\Bigl(q_1\cdots q_{i-1} \prod_{p_s<q_i}p_s  \Bigr)
\end{equation}
or 
\begin{equation}\label{vio2}
p_i>\theta\Bigl(p_1\cdots p_{i-1} \prod_{q_s<p_i}q_s\Bigr).
\end{equation}
First assume \eqref{vio1}. If $m=\prod_{p_s<q_i}p_s$, then $q_i>\theta(m)$, so $\theta(m)<P^+(q)$. 
If $\prod_{p_s<q_i}p_s < m$, then $\prod_{p_s<q_i}p_s=p_1\cdots p_{r-1}$, say. We have
$$
p_r \ge q_i > \theta(p_1\cdots p_{r-1} q_1\cdots q_{i-1})\ge \theta_q(p_1\cdots p_{r-1}),
$$
so $m \notin  \mathcal{B}_{\theta_q} $. 

If \eqref{vio2} holds, then 
$$
p_i>\theta\Bigl(p_1\cdots p_{i-1} \prod_{q_s<p_i}q_s\Bigr) \ge  \theta_q(p_1\cdots p_{i-1}).
$$
The last inequality clearly holds if $q=\prod_{q_s<p_i}q_s$. If $q>\prod_{q_s<p_i}q_s=q_1\cdots q_{j-1}$,
say, then the last inequality follows from $q_j\ge p_i$. 
Thus, we have $m \notin  \mathcal{B}_{\theta_q} $. 
\end{proof}

\section{Proof of Corollary \ref{cor0}}
Let $d=\gcd(q,a)$. If $n\equiv a \bmod q$, we write $q=d \tilde{q}$, $a=d\tilde{a}$ and 
$n=d\tilde{n}$. 
By Lemma \ref{lemprop1},
$$
\sum_{n \in \mathcal{B}(x) \atop n\equiv a \bmod q} 1
=\sum_{d\tilde{n} \in \mathcal{B}(x) \atop \tilde{n}\equiv \tilde{a} \bmod \tilde{q}} 1
=O(\Psi(x/d,d))
+\sum_{\tilde{n} \in \mathcal{B}_{\theta_d}(x/d) \atop \tilde{n}\equiv \tilde{a} \bmod \tilde{q}} 1.
$$
Since $\gcd(\tilde{q},\tilde{a})=1$, Theorem \ref{thm1} shows that the last sum is
$$
 O_A\left(\frac{x/d}{(\log(x/d))^A}\right)+ \frac{1}{\varphi(\tilde{q})} 
 \sum_{\tilde{n} \in \mathcal{B}_{\theta_d}(x/d) \atop \gcd(\tilde{n},\tilde{q})=1} 1 .
$$
Appealing again to Lemma \ref{lemprop1}, we have
$$
\sum_{\tilde{n} \in \mathcal{B}_{\theta_d}(x/d) \atop \gcd(\tilde{n},\tilde{q})=1} 1 
=O(\Psi(x/d,d))+ \sum_{\tilde{n}d \in \mathcal{B}_{\theta}(x) \atop \gcd(\tilde{n},\tilde{q})=1} 1 
=O(\Psi(x/d,d))+\sum_{n \in \mathcal{B}_{\theta}(x) \atop \gcd(n,q)=d} 1 .
$$
Since $d \le q \le (\log x)^A$, we have $\Psi(x/d,d)\le \Psi(x,q) \ll_A x/(\log x)^A$, by \eqref{Req}.

\section{Proof of the last statement of Corollary \ref{cor3}}\label{seccqa}

If $B(x,q,a)\le 1$ for all $x$, then $c_{q,a}=0$ by \eqref{BxqaAsymp}.

If $B(x,q,a) \ge 2$ for some $x$, then there exists an $m\in \mathcal{B}$, $m>q$, with $m\equiv a \bmod q$. 
Consider the set 
$$
\mathcal{N}(x)=\{ n\le x: n\equiv 1 \bmod q, n \in \mathcal{D}_{q+1} \},
$$
where $\mathcal{D}_t$ denotes the set of $t$-dense integers. We claim that (i)
$|\mathcal{N}(x)| \gg_q x/\log x$ and that (ii) $mn \in \mathcal{B}$ for all $n\in \mathcal{N}(x)$. 
Since $mn\equiv a \bmod q$, these two claims imply $c_{q,a}>0$. 
By Theorem \ref{thm1}, the first claim is equivalent to 
$$
|\{ n\le x: \gcd(n,q)=1,  n \in \mathcal{D}_{q+1} \}| \gg_q \frac{x}{\log x},
$$
which follows from \cite[Lemma 5]{PW}.

To show the second claim, assume $m\in \mathcal{B}$, $m>q$, $n\in \mathcal{N}(x)$, and write $m=r_1\cdots r_l$, $n=p_1\cdots p_k$, the prime factorizations in increasing order. 
Since $m\in \mathcal{B}$, we have
$$
r_{i+1}\le \theta(r_1 \cdots r_i) \le \theta\Bigl(r_1 \cdots r_i \prod_{p_j<r_{i+1}}p_j\Bigr) \quad (0\le i < l).
$$
If $r_j< p_{i+1}\le r_{j+1}$, for some $i<k$, $j<l$, then
$$
p_{i+1}\le r_{j+1}\le \theta(r_1\cdots r_j) \le \theta (r_1 \cdots r_j  \, p_1 \cdots p_i ) .
$$
If $p_{i+1} > r_l$, then $n\in \mathcal{D}_{q+1}$ implies
$$
p_{i+1} \le (q+1)p_1\cdots p_i \le m p_1\cdots p_i \le \theta(m p_1\cdots p_i).
$$
This shows that $mn \in \mathcal{B}$. 

\section*{Acknowledgments}
The author thanks John Friedlander for pointing out \cite{Fried}, Felix Weingartner for sharing his computer to complete the
calculations leading to Table \ref{table1}, and the anonymous referee for helpful suggestions.

\end{document}